\documentclass[11pt]{amsart}  
\usepackage{graphicx}
\usepackage{latexsym}
\usepackage{epstopdf}
\usepackage{amsfonts,amsmath,amssymb}

\newtheorem{lemma}{Lemma}
\newtheorem{example}{Example}
\newtheorem{theorem}{Theorem}
\newtheorem{proposition}{Proposition}

\newtheorem{definition}{Definition}

\begin{document}

\author[M. B\'ona]{Mikl\'os  B\'ona}
\title[A Bijective Proof of an Identity Extending a 
Result of Haj\'os]{A Bijective Proof of an Identity Extending a 
Classic Result of Haj\'os}
\address{\rm Mikl\'os B\'ona, Department of Mathematics, 
University of Florida,
358 Little Hall, 
PO Box 118105, 
Gainesville, FL 32611--8105 (USA)
}

\date{}

\begin{abstract} 
We provide bijective proofs of two classic identities
that are very simple to prove using generating functions, 
but surprisingly difficult to prove combinatorially. The problem
of finding a bijective proof for the first identity was first raised
in the 1930s. The second, more involved identity takes the first one
a step further. 
\end{abstract}

\maketitle

\centerline{{\em Dedicated to the memory of Herb Wilf}}

\section{Introduction}
\subsection{A Classic Problem}
The obvious identity 
\[\frac{1}{1-4x} = \frac{1}{\sqrt{1-4x}} \cdot  \frac{1}{\sqrt{1-4x}} \]
implies, after equating coefficients of
$x^n$ on both sides, the identity
\begin{equation} \label{fullsoccer}
4^n = \sum_{i=0}^{n} {2i\choose i}{2(n-i)\choose n-i}.\end{equation} 
Finding a bijective proof for (\ref{fullsoccer}) is surprisingly difficult.

Note that (\ref{fullsoccer}) has the following simple interpretation.
There are as many ways for two teams to score a total of $2n$ goals in a 
soccer game as there are ways for those two teams to reach a tie score
at the end of the first half of the soccer game, then to hold an 
intermission, and then to reach an $n-n$ tie at the end of the game.

Richard Stanley (\cite{ejc1}, page 52)
 provides a brief history of the quest for a 
combinatorial proof of (\ref{fullsoccer}), mentioning that the problem
was raised by P\'al Veress and solved by Gy\"orgy Haj\'os in the 1930s.
 Stanley  also mentions a more recent combinatorial proof by Daniel Kleitman 
\cite{kleitman} and a discussion by M\'arta Sv\'ed \cite{sved}.  

As a ``warmup'' to our main result, in this paper we will provide a
combinatorial proof of (\ref{fullsoccer}) by constructing a simple 
bijection.  This result, and its proof,  while similar in nature
to our main result, are not necessary for understanding that main result.
Readers so inclined may skip Section \ref{warmup}. 

\subsection{Our Main Result: An Identity One Step Further} 
The obvious identity 
\[\frac{1}{(1-4x)^{3/2}} = \frac{1}{\sqrt{1-4x}} \cdot  \frac{1}{\sqrt{1-4x}} 
\cdot \frac{1}{\sqrt{1-4x}}\]
implies, after equating coefficients of
$x^n$ on both sides, the identity
\begin{equation} \label{fullhockey}
(2n+1)\cdot {2n\choose n} = \sum_{i+j+k=n}^{n} {2i\choose i}{2j\choose j}
{2k\choose k},\end{equation}
where the sum on the right-hand side is taken over all triples of 
non-negative integers that sum to $n$. 

Note that (\ref{fullhockey}) has the following interpretation. There are
as many ways for two soccer teams to  reach any score in the first half,
then hold an intermission, and then reach an $n-n$ tie at the end of the game
 as there are
 ways for two ice hockey teams
to reach a tie score at the end of the first period, hold an intermission,
reach a tie score at the end of the second period, hold an intermission, 
and then reach a score of $n-n$ at the end of the game.
 (Soccer games have two halves, 
while ice hockey games have three periods. Two outcomes are considered 
identical if they agree in the score at the end of each half or period.)

In Section \ref{main}, we give a bijective proof of (\ref{fullhockey}).
The proof will have an interesting structure. After disposing of some
trivial cases, we will construct two bijections, each of which will 
map {\em half} of the remainder of the domain onto {\em half} of the 
remaining range of the bijection that we intend to build. However, 
these two halves will not be disjoint. Fortunately, there will be a simple
bijection between the intersection of these two halves and the complement
 of their union, which will enable us to complete our proof.

\section{Warming Up. A Bijective Proof of the First Identity} \label{warmup}

In this section, we will work with {\em northeastern lattice paths}. 
A northeastern lattice path is a lattice path consisting of 
 steps $(0,1)$ (a north
step) and $(1,0)$ (an east step). The number of northeastern lattice
paths from $(0,0)$ to $(a,b)$ is clearly ${a+b\choose a}$ for all non-negative
integers $a$ and $b$, whereas the number of northeastern lattice paths starting
at $(0,0)$ and consisting of $2m$ steps is $2^{2m}=4^m$. 

Therefore, identity (\ref{fullsoccer}) that we are going to prove in this
section claims that the number of  {\em all} northeastern lattice paths
starting at $(0,0)$ and consisting of $2n$ steps agrees with the number of
northeastern lattice paths that start at $(0,0)$, end in $(n,n)$, and
are marked at some point $(i,i)$ on the main diagonal. 

The observation we present in this paragraph was made by Kleitman in 
\cite{kleitman}. A northeastern lattice path $p$ starting at $(0,0)$ and 
consisting of $2n$ steps has at least one point in common
 with the main diagonal 
$x=y$ (the point $(0,0)$ is always such a point).
 Let us say that $N=(i,i)$ is the last (most northeastern) such point, and
let $N$ split $p$ into parts $p_1$ and $p_2$. Then $p_2$ is a northeastern
lattice path starting on the main diagonal $x=y$ that {\em never touches}
the main diagonal again, while $p_1$ starts and ends in the main diagonal. 
Therefore, if we can find a bijection $b$ from the set of all northeastern
lattice paths starting at $(i,i)$ and ending in $(n,n)$
 to the set of such $p_2$,
then the map $p\rightarrow (p_1,b^{-1}(p_2))$ will be a bijection proving 
(\ref{fullsoccer}). Here $(p_1,b^{-1}(p_2))$ means the concatenation of $p_1$
 and
$b^{-1}(p_2)$, marked at the point where these two paths meet. 

We will construct such a bijection in the following theorem. 

\begin{theorem} \label{soccer}
There is a natural bijection $F_n$ from the set $X_n$
 of lattice paths with $2n$ steps  that start 
at $(0,0)$ and end in $(n,n)$ to the set $Y_n$ of 
lattice paths that start at $(0,0)$, consist of $2n$ steps
and never touch the $x=y$ line other than in their starting point. 
\end{theorem}

Note the following interpretation of Theorem \ref{soccer}. 
There are as many ways for a soccer game to end in an $n-n$ tie
as there are ways for a soccer game to bring a total of $2n$ goals
so that the score is {\em never} a tie other than at 0-0. 

The following lemma is the main ingredient of our proof of Theorem
\ref{soccer}.

\begin{lemma} \label{steps} Let $k$ and $n$ be positive integers so that
$n\leq k\leq 2n$.
Let $A(n,k)$ be the set of all lattice paths from $(1,0)$ to $(k,2n-k)$.
Let $B(n,k)$ be the set of all lattice paths from $(1,0)$ to $(k,2n-k)$ 
that never touch the line $x=y$.

Then there exists a natural bijection $f_{n,k}:A(n,k)\rightarrow
B(n,k) \cup A(n,k+1)$.
\end{lemma}

\begin{proof} (of Lemma \ref{steps})
Let $p\in A(n,k)$. If $p\in B(n,k)$, then we set $f_{n,k}(p)=p$. Otherwise,
let $P$ be the first point on $p$ that is also on the line $x=y$.
Let us reflect the part of $p$ that is between $(1,0)$ and $P$ through
the line $x=y$, to get the path $p'$. Then $p'$ is a path from 
$(0,1)$ to $(k,2n-k)$. Translate $p'$ by the vector $(1,-1)$, to get the path
$p''$ that is a path from $(1,0)$ to $(k+1,2n-k-1)$. Then set $f_{n,k}(p)=p''$.

See Figure \ref{stepsfig} for an illustration. 

It is straightforward to see that $f_{n,k}$ is a bijection. Indeed, each
element $q$ of $A(n,k+1)$, when translated by the vector $(-1,1)$, will 
intersect
the line $x=y$, so the unique $f_{n,k}^{-1}(q)$ can be easily recovered. 
\end{proof}

\begin{figure}[ht]
 \begin{center}
  \includegraphics[width=60mm]{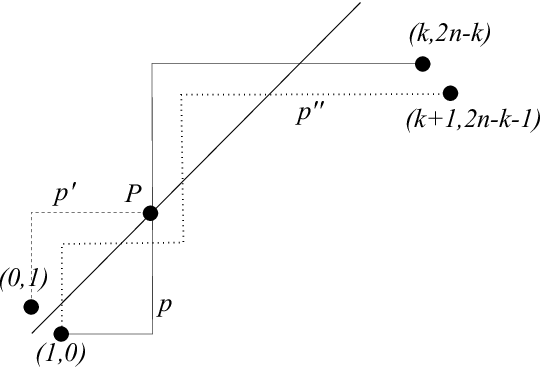}
  \label{stepsfig}
\caption{Obtaining $f(p)=p''$ for a path $p$ intersecting the line $x=y$.
Note that after the point $P$, the paths $p$ and $p'$ coincide. }
 \end{center}
\end{figure}

Now we are in a position to prove Theorem \ref{soccer}.

\begin{proof} (of Theorem \ref{soccer})
Let $p\in X_n$, and let us assume without loss of generality that 
$p$ starts with an east step. Keep the notation of the proof of 
Lemma \ref{steps}. Then, identifying $p$ with the path obtained from $p$
by removing its predetermined first step, $p$ 
 is an element
of $A(n,n)$.
So $f_{n,n}(p)=p_1$ is defined, and $p_1\in B(n,n)\cup A(n,n+1)$.

If $p_1\in B(n,n)$, then we set $F_n(p)=p_1$ and we stop. Otherwise, that is,
if    $p_1\in A(n,n+1)$, then we repeat this procedure. More precisely,
we consider $p_2=f_{n,n+1}(p_1)\in B(n,n+1)\cup A(n,n+2)$. If $p_2\in 
B(n,n+1)$, then we set $F_n(p)=p_2$ and we stop, while
 if $p_2\in A(n,n+2)$, then we continue
by considering $p_3=f_{n,n+2}(p_2)\in  B(n,n+1)\cup A(n,n+2)$.

If we have not stopped by the $i$th step, then in the $i$th step, we consider
$p_i=f_{n,n+i-1}(p_{i-1})\in B(n,n+i-1)\cup A(n,n+i)$. If $p_i\in   B(n,n+i-1)$,
then we set $F_n(p)=p_i$ and we stop. If $p_i \in A(n,n+i)$, then we 
go to step $i+1$, where we consider $p_{i+1}=f_{n,n+i}(p_i)$ according to the
same rules. The procedure stops and $F_n(p)$ gets defined when 
$p_i\in B(n,n+i-1)$, that is, when the current image of $p$ stays below the
main diagonal. 

See Figure \ref{soccergen} for an illustration.
\begin{figure}[ht]
 \begin{center}
  \includegraphics[width=60mm]{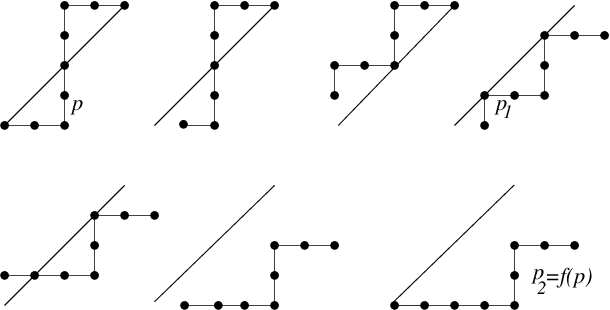}
  \label{soccergen}
\caption{All steps of computing $f(p)$, where $p=EENNNNEE$. }
 \end{center}
\end{figure}

This procedure will define $F_n(p)$ for all $p\in X_n$ starting
with an east step. Indeed, note that
$A(n,2n)=B(n,2n)$, so all paths $A(n,2n)$ are below the main diagonal. 
So in at most $n$ steps, $F_n(p)\in Y_n$ will be defined for all $p\in X_n$
that start with an east step. 

If $p$ starts with a north step, simply reflect $p$ through the line $x=y$,
to get $p^t$. Then proceed as above, to obtain $F_n(p^t)$. Finally, reflect
$F_n(p^t)$ through the line $x=y$ to get $F_n(p)$. Note that in this case,
 $F_n(p)$ will be entirely above the main diagonal, except for its starting
 point.

We will now show that $F_n$ is a bijection by showing that it has an inverse.
Let $u\in Y_n$,
and let us assume without loss of generality that $u$ starts with an east
step. That means that $u$ ends below the main diagonal, in the point 
$(m,2n-m)$ for some $m>n$. So $u\in B(n,m)$, and in particular,
$u\in A(n,m)$. Therefore, the unique path
$u_1=f_{n,m-1}^{-1}(u)\in A(n,m-1)$
 exists. In order to get $u_1$, translate $u$ by
the vector $(-1,1)$. Let $U$ be the first intersection point of the line 
$x=y$ and the translated copy of $u$. Then reflect the part of $u$ that is
before $U$ through the line $x=y$ to get $u_1$. 
Then proceed in a similar manner. That is, since $u_1\in A(n,m-1)$, the unique
path $u_2=f_{n,m-2}^{-1}(u)\in A(n,m-2)$
 exists, then the unique path $u_3=f_{n,m-3}^{-1}(u)\in A(n,m-3)$ exists, and
so on. After $m-n$ steps, we get a path $u_{m-n}\in A(n,n)$, and that is 
the unique path satisfying $F_n(u_{m-n})=u$.
\end{proof}

See Figure \ref{diagram} for the algorithm that defines  the map $F_n$.
\begin{figure}[ht]
 \begin{center}
  \includegraphics[width=70mm]{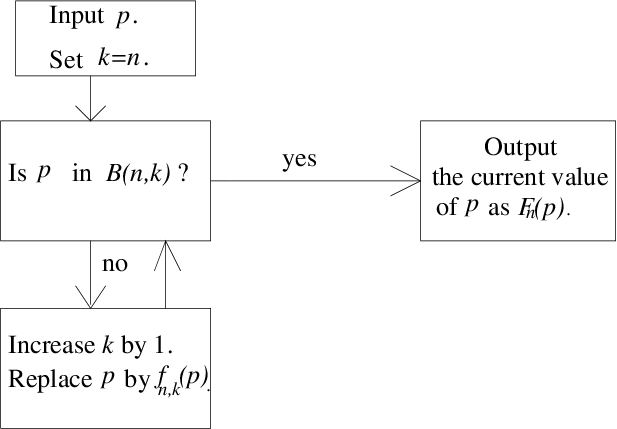}
\caption{The algorithm that produces $F_n(p)$. Note that the algorithm
always stops after at most $n$ increases of $k$.}
\label{diagram}
 \end{center}
\end{figure}

\section{Our Main result: A Bijective Proof of the Second Identity} \label{main}
The goal of this section is to give a bijective proof of formula
(\ref{fullhockey}). 

The lattice paths we use in this section will have two allowed steps,
$(1,1)$ (an ``up'' step) and $(1,-1)$, a ``down'' step. For the rest
of this paper, when we say ``lattice path'', we mean a lattice path of that
kind.  

Let $T_n$ be the set of ordered triples of lattice paths 
 $P=(A,B,C)$, where $A$ is a lattice path consisting of $i$ up steps
and $i$ down steps, 
 $B$ is a lattice path consisting of $j$ up steps and $j$ down steps,
and $C$ is a lattice path consisting of $k$ up steps and $k$ down
steps, for some ordered triple $(i,j,k)$ of non-negative integers satisfying
$i+j+k=n$. Clearly, $|T_n|=\sum_{i+j+k=n}{2i\choose i}
{2j\choose j}{2k\choose k}$, which is the left-hand side
of the identity that we intend to prove. We will place our triples $(A,B,C)$
so that $A$ starts at $(0,0)$; then $C$ will end at $(2n,0)$.

Let $D_n$ be the set of ordered pairs $(H,X)$, where $H$ is a lattice
path consisting of $n$ up steps and $n$ down steps, and $X$ is one of the
$(2n+1)$ lattice points on $H$. It is then clear that $|D_n|=(2n+1)
{2n\choose n}$, which is the right-hand side of the identity to be
proved. Again, we will place the elements of $D_n$ so that $H$ starts at 
$(0,0)$ and ends at $(2n,0)$.

We are going to define a bijection $g:T_n\rightarrow D_n$ in three steps,
the ``trivial step'', the ``core step'', and the ``correction step''.

\subsection{The Trivial Step}
Let $R_n\subseteq T_n$ denote the subset of $T_n$ that consists of pairs
$(A,B,C)$ in which $C$ is empty. Let $D_n^0\subseteq D_n$ be the subset
of $D_n$ consisting of pairs $(H,X)$ in which $X$ is on the
{\em horizontal axis}.  For elements of $R_n$, we simply define
$r(A,B,C)=r(A,B,\emptyset)$ to be the pair $(H,X)\in D_n^0$ in which
$H$ is simply the concatenation of $A$ and $B$, and $X$ is the point
at which $A$ and $B$ meet.
The following proposition is then obvious.
\begin{proposition}
The map $r:R_n \rightarrow D_n^0$ defined in the previous paragraph
is a bijeciton.
\end{proposition}

\subsection{The Core Step}
First, we define a map on a large subset (almost half) of $T_n$. 
Let $U_n$ be the subset of $T_n$ that consists of elements $(A,B,C)\in T_n$
so that (exactly) one of the following two conditions hold:
\begin{enumerate}
\item The path 
$B$ ends with a down step and the path $C$ has at least one point strictly above
the horizontal axis, or 
\item The path $B$ is empty,  the path $C$ has at least one point strictly above
the horizontal axis.
\end{enumerate}

Let $D_n^+$ denote the subset of $D_n$ consisting of pairs $(H,X)$
in which $X$ is a point of $H$ that is {\em strictly above} the horizontal
axis.

Let $s:U_n\rightarrow D_n^+$ be defined as follows. The {\em height} of
a lattice point is simply its vertical coordinate. Let $(A,B,C)\in T_n$ and
let us assume first that the first criterion is satisfied, that is,
 $B$ ends with a down step and
$C$ has at least one point strictly above the horizontal axis.
Let $K$ be the leftmost point on $C$ that has maximal height. Let this
height be $k$, then we it follows from our assumptions  that $k>0$. 
Let $K$ cut $C$ into the left part $C_1$ and the right part $C_2$. 
Now we define $s(A,B,C)=(H,X)$, where $H$ is the concatenation of the paths
$C_1$, $A$, $B$, and $C_2$ (each path starts at the endpoint of the preceding
path), and where $X$ is the point where $A$ and $B$ meet. 
See Figure \ref{pathsup} for an illustration.

If $B$ is empty, then process is identical to the one above. Note that because
$B$ is empty, $X$ will be the starting point of $C_2$.

Note that in both cases, $X$ has the same height $k$ as $K$, and that
$K$ is the leftmost point of $H$ on the horizontal line $y=k$. In the
special case when $A$ is
empty, the points $X$ and $K$ coincide.

\begin{example}
Figure \ref{pathsup} illustrates how the map $s$ is defined. We drew 
the parts of $H$ that are defined by $A$, $B$, and $C$, with the respective
line styles of $A$, $B$, and $C$. 
 
\begin{figure}[ht]
 \begin{center}
   \includegraphics[width=60mm]{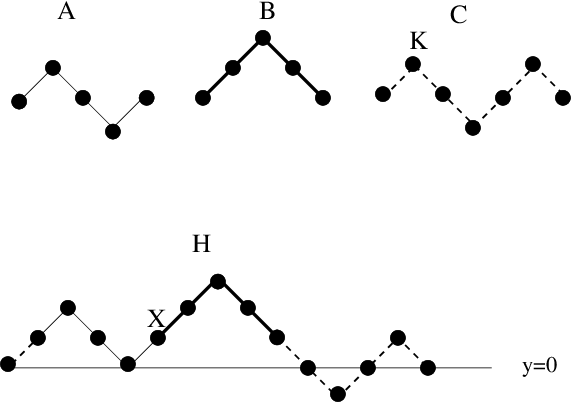}
\caption{The map $s:U_n\rightarrow D_n^+$. }
\label{pathsup}
 \end{center}
\end{figure}
\end{example}

\begin{lemma} \label{sisbij}
The map $s:U_n \rightarrow D_n^+$ defined above is a bijection.
\end{lemma}

\begin{proof}
Let $(H,X)\in D_n^+$. We will show that $(H,X)$ has exactly  one preimage
under $s$. Draw a horizontal line $\ell$ 
across $X$. By the paragraph immediately
preceding our lemma, the point $K$ can be recovered as the leftmost common
point of $\ell$ and $H$. Therefore, if $s(A,B,C)=(H,X)$, then $A$ can
be recovered as the (possibly empty) part of $H$ that is between $K$ and $X$.
Then $C_1$ can be recovered as the part of $H$ that is between $(0,0)$
and $K$. 

We still have to show that there is exactly one way in which we can split
the part of $H$ that is between $X$ and $(2n,0)$ into $B$ and $C_2$.
Note that it follows from our definitions that if such $B$ and $C_2$
exist, and $B$ is not empty, then $B$ ends with a down step and $C_2$ starts
 with a down step.
Indeed, $C$ is split into two parts $C_1$ and $C_2$
 at a point of {\em maximal height}, so
$C_2$ never goes above its starting height. 
In particular, this means that if $B$ is non-empty,
then the point $Z$ where 
 $B$ and $C_2$ meet is
the {\em rightmost} point of $\ell$ strictly on the right of $X$ 
that is crossed by $H$, (that is, $H$ 
reaches $Z$ by a down step and leaves it by another down step). So if such
$Z$ exists, then we recover $C_2$ and $B$ as the parts of $H$ between 
$X$ and $Z$ and $Z$ and $(2n,0)$. If such $Z$ does not exist, then $B$ 
is empty, and $C_2$ is the part of $H$ that is between $X$ and $(2n,0)$. 
\end{proof}

\begin{example}
Figure \ref{fig:inverseup} illustrates how to recover $s^{-1}(H,X)$ if
$(H,X)\in D_n^+$ is given.
 
\begin{figure}[ht]
 \begin{center}
   \includegraphics[width=60mm]{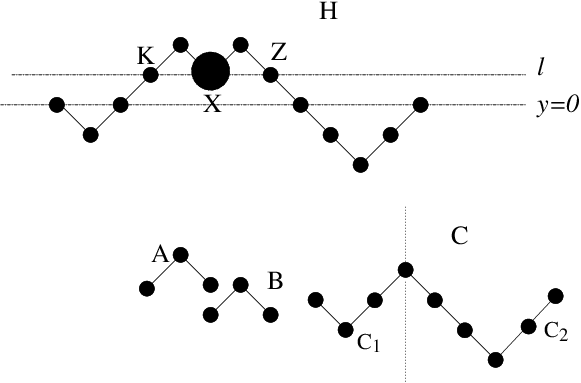}
\caption{Finding the inverse image of $(H,X)\in D_n^+$. }
\label{fig:inverseup}
 \end{center}
\end{figure}
\end{example}

Now let $V_n$ be 
the set of triples $(A,B,C)\in T_n$ that are reflected images of
the elements of $U_n$ through the horizontal axis. That is,
$(A,B,C)\in V_n$ if (exactly) one of the following conditions hold.
\begin{enumerate}
\item The path $B$ ends with an up step and the path
$C$ has a point strictly below the
horizontal axis, or 
\item the path $B$ is empty and the path $C$ has a point strictly below
the horizontal axis. 
\end{enumerate}

Now we define a map $t:V_n\rightarrow D_n^-$ which is essentially the reflected 
image of $s$ through the horizontal axis.

Let $(A,B,C)\in V_n$. Let $k<0$ be the minimum height that any point of $C$
has, and let $K$ be the {\em leftmost} point of $C$ with that height. 
Let $K$ cut $C$ into the left part $C_1$ and the right part $C_2$. 
Define $z(A,B,C)=(H,X)$, where $H$ is the concatenation of the paths
$C_1$, $A$, $B$, and $C_2$ and where $X$ is the point where $A$ and $B$ meet. 

Again, 
if $B$ is empty, then process is identical to the one above. Note that because
$B$ is empty, $X$ will be the starting point of $C_2$.

Just as we saw it when we defined $s$, the point
 $X$ has the same height $k$ as $K$, and
$K$ is the leftmost point of $H$ on the horizontal line $y=k$. In the
special case when $A$ is
empty, the points $X$ and $K$ coincide.

\begin{lemma} \label{tisbij}
The map $t:V_n\rightarrow D_n^{-}$ is a bijection.
\end{lemma}

\begin{proof} This can be proved in a way analogous  to the way
in which  Lemma \ref{sisbij} was proved. Alternatively, note that
if $r_x$ denotes reflection through the horizontal axis, then
\[t=  (r_x\circ s) \circ r_x,\]
and the proof is immediate since all three maps on the right-hand side
are bijective. 
\end{proof}

\subsection{The Correction Step}
Our work is not finished, since $U_n$ and $V_n$ are not disjoint, and
also, $U_n\cup V_n \neq T_n - R_n$. (Two halves cannot cover
a whole if they overlap.) Let $I_n=U_n\cap V_n$, and let
$J_n=T_n-(U_n\cup V_n)$, that is, the set of elements of $T_n$ that
are not elements of $U_n\cup V_n$. 

It follows from our definitions that the following two
propositions hold. 

\begin{proposition} \label{whatisi} 
The set $I_n\subset T_n$ is the set of triples $(A,B,C)$ such that
 $B$ is empty and $C$ has both a point strictly above
the horizontal axis and a point strictly below the horizontal axis.
\end{proposition}

\begin{proposition} \label{whatisj}
The set $J_n\subset T_n$ is the set of triples $(A,B,C)$ such that
\begin{enumerate}
\item the path $B$ ends with a down step and the path $C$ is not
empty, and the path $C$ is entirely
below the horizontal axis, or
\item  the path $B$ ends with an up step and the path $C$ is not empty,
and the path $C$ is entirely
above the horizontal axis.
\end{enumerate}
\end{proposition}

Fortunately, there is a simple bijection mapping $J_n$ to $I_n$. 
Let $(A,B,C)\in J_n$, and let us assume without loss of generality that
$(A,B,C)$ satisfies the first criterion of Proposition \ref{whatisj}, 
that is, $C$ is not empty,  $C$ is
 entirely below the horizontal axis, and $B$ ends with
a down step. Let $z(A,B,C)=(A,\emptyset,BC)$, that is, 
we get $z(A,B,C)$ by prepending $C$ by $B$, and thus making
the middle path of the image empty. Note that $z(A,B,C)\in I_n$, 
since the middle path of $z(A,B,C)$ is empty, and its last
path, $BC$ has a point strictly above the horizontal axis (since
$B$ ends with a down step, and a point strictly below the horizontal
axis, since $C$ has such a point. 
If $(A,B,C)$ satisfies the second criterion of Proposition
\ref{whatisj}, the definition of $z(A,B,C)$ is identical to the
above, and the proof of the fact that $z(A,B,C)\in I_n$ is 
analogous. 

See Figure \ref{fig:correction} for an illustration.
\begin{figure}[ht]
 \begin{center}
   \includegraphics[width=60mm]{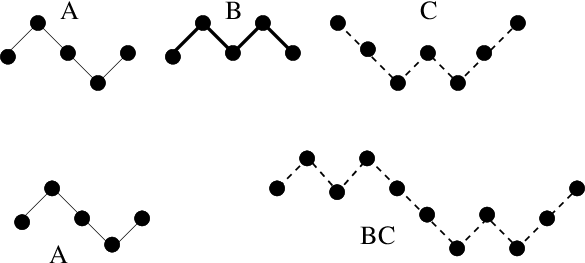}
\caption{The map $z:J_n\rightarrow I_n$. }
\label{fig:correction}
 \end{center}
\end{figure}

\begin{lemma}
The map $z:J_n\rightarrow I_n $ defined above is a bijection.
\end{lemma}

\begin{proof}
Let $(P,\emptyset,Q)\in I_n$. By definition, $R$ has a point
strictly below, and a point strictly above, the horizontal 
axis. So there is at least one point  that is common
to $R$ and the horizontal axis at which $R$ {\em crosses}
the horizontal axis, from one side to the other.
Let $Y$ be the {\em rightmost} such point. Then the unique
element $z^{-1}(P,\emptyset,Q)=(A,B,C)$ can be found by setting $A=P$, 
setting $C$ to be the part of $Q$ between $Y$ and the end of $Q$, 
and setting $B$ to be the part of $Q$ between the start of
$Q$ and $Y$. 
\end{proof}

Figure \ref{fig:inversecorr} illustrates the map $z^{-1}$. 

\begin{figure}[ht]
 \begin{center}
   \includegraphics[width=60mm]{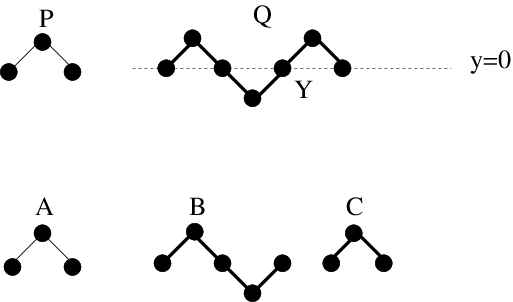}
\caption{Finding the inverse image of $(P,\emptyset, Q)\in I_n$. }
\label{fig:inversecorr}
 \end{center}
\end{figure}

\subsection{The Composite Bijection}
We are now in position to define the bijection that was our
goal to construct in this section.

\begin{definition} \label{defofmaster}
Let $n$ be any positive integer. Let $(A,B,C)\in T_n$, 
and let $g(A,B,C)$ be equal to 
\begin{itemize}
\item $r(A,B,C)$ if $(A,B,C)\in R_n$,
\item $s(A,B,C)$ if $(A,B,C)\in U_n$,
\item $t(A,B,C)$ if $(A,B,C)\in V_n - U_n$, 
\item $t(z(A,B,C))$ if $(A,B,C)\in J_n$.
\end{itemize}
\end{definition}

\begin{theorem}
The map $g:T_n\rightarrow D_n$ given in 
Definition \ref{defofmaster} is a bijection.
\end{theorem}

\begin{proof}
We show that $g$ has an inverse. Let $(H,X)\in D_n$. If $(H,X)\in D_n^0$,
then $g^{-1}(H,X)=r^{-1}(H,X)$. If $(H,X)\in D_n^+$,
then $g^{-1}(H,X)=s^{-1}(H,X)$. Finally, if  $(H,X)\in D_n^{-}$, then we need
to consider $t^{-1}(H,X)$. If $t^{-1}(H,X)\in V_n-U_n$, 
then $g^{-1}(H,X)=t^{-1}(H,X)$, and if $t^{-1}(H,X) \in  I_n=U_n\cap V_n$, 
then $g^{-1}(H,X)=z^{-1}(t^{-1}(H,X))$. This shows that each element of $D_n$
has exactly one preimage under $g$, proving that $g$ is a bijection. 
\end{proof}

\section{A simple generalization}
Finally, we point out an easy generalization of the result in 
Section \ref{warmup}. Consider the trivial identity
\begin{equation} \label{even}
\left(\frac{1}{\sqrt{1-4x}} \right) ^{2m} = (1-4x)^{-m}.\end{equation}
Equating coefficients of $x^n$ leads to the identity
\begin{equation}
\label{evenbinom}
\sum_{k_1+k_2+\cdots +k_{2m}=n} \prod_{i=1}^{2m} {2k_i\choose k_i} = 
{m+n-1\choose m-1} 4^n,\end{equation}
where the sum on the left-hand side is taken over all $2m$-tuples 
$(k_1,k_2,\cdots ,k_{2m})$ of non-negative integers that sum to $n$.

A combinatorial proof of (\ref{evenbinom}) is easy to obtain, since both
sides count all northeastern lattice paths starting at (0,0) and consisting
of $2n$ steps that have been marked at $m-1$ points so that if $(x,y)$ denotes
the coordinates of a marked point, then both $x$ and $y$ are integers, 
and $x+y$ is even. It is possible that one point is marked several times. 

Indeed, it is obvious that the right-hand side of (\ref{evenbinom}) counts 
such paths, since each path can be marked at $n+1$ points (counting 
the starting and the ending points), and the same point can be marked several
times.  

Now note that between any two consecutive marked points of a path, 
a northeastern lattice path of even length is formed. Turn each such path
into a pair of paths as seen in the proof of (\ref{fullsoccer}). This will
show
that the left-hand side of (\ref{evenbinom})  counts the same paths as
the right-hand side.
  
Note that the case of $m=2$ enables one to make a statement about professional
basketball games, which consist of  four quarters.

\end{document}